\documentclass[12pt, english, a4paper]{amsart}

\usepackage{graphicx}
\usepackage{dsfont}

\usepackage{amsmath,amssymb}
\usepackage{bm}
\usepackage{euscript,color}
\usepackage{babel}
\usepackage{amsthm}
\usepackage{amsfonts}
\usepackage{latexsym}
\usepackage{fancyhdr}
\usepackage{enumerate}

\usepackage[left=2.5cm, top=3cm, bottom=3cm, right=2.5cm]{geometry}

\newcommand{\R}{\mathds{R}}
\newcommand{\C}{\mathds{C}}
\newcommand{\Z}{\mathds{Z}}

\newcommand{\Ric}{\mathop{\mathrm{Ric}}\nolimits}

\makeatletter
\renewcommand{\section}{\@startsection%
{section}% name
{1}% level
{0mm}% indent
{1.5\bigskipamount}% beforeskip
{0.5\bigskipamount}% afterskip
{\centering\normalsize\sc}}% style

\renewcommand{\paragraph}{\@startsection%
{paragraph}% name
{4}% level
{0mm}% indent
{\bigskipamount}% beforeskip
{-1.25ex}% afterskip
{\normalsize\sl}}% style

\def\provedboxcontents#1{$\square$}

\makeatother

\newtheoremstyle{thm}{}{}{\slshape}{}{\scshape}{.}{0.5em}{}

\newtheoremstyle{def}{}{}{}{}{\scshape}{.}{0.5em}{}

\newtheoremstyle{rmk}{}{}{}{}{\scshape}{.}{0.5em}{}

\newtheoremstyle{claim}{}{}{}{}{\slshape}{.}{0.5em}{}

\theoremstyle{thm}
\newtheorem{newstatement}{newstatement}
\newtheorem{lemma}{Lemma}
\newtheorem{theorem}[newstatement]{Theorem}
\newtheorem{corollary}{Corollary}

\newtheorem*{conjecture*}{Conjecture}

\theoremstyle{def}

\newcommand{\K}{K\"{a}hler}

\theoremstyle{rmk}
\newtheorem{remark}{Remark}

\theoremstyle{claim}

\expandafter\let\expandafter\oldproof\csname\string\proof\endcsname
\let\oldendproof\endproof
\renewenvironment{proof}[1][\proofname]{%
  \oldproof[\slshape #1]%
}{\oldendproof}

\let\geq\geqslant
\let\leq\leqslant

\let\epsilon\varepsilon

\renewcommand{\emph}[1]{{\slshape #1}}
\renewcommand{\em}{\sl}

\title{Ricci flat Calabi's metric is not projectively induced}

\author{Andrea Loi}
\address[Andrea Loi, Fabio Zuddas]{Dipartimento di Matematica e Informatica, Universit\`{a} di Cagliari,
Via Ospedale 72, 09124 Cagliari, Italy}
\email{loi@unica.it\\ fabio.zuddas@unica.it}
\author{Michela Zedda}
\address[Michela Zedda]{Dipartimento di Scienze Matematiche, Fisiche e Naturali, Parco Area delle Scienze 53/A  Parma (Italy)}
\email{michela.zedda@unipr.it}
\author{Fabio Zuddas}

\thanks{
The first and the third authors were supported  by Prin 2015 -- Real and Complex Manifolds; Geometry, Topology and Harmonic Analysis -- Italy,  by GESTA - Funded by Fondazione di Sardegna and Regione Autonoma della Sardegna and by KASBA- Funded by Regione Autonoma della Sardegna.\\
Finally, all the three authors were supported by INdAM GNSAGA - Gruppo Nazionale per le Strutture Algebriche, Geometriche e le loro Applicazioni.
}
\date{}

\begin{document}
\maketitle
\begin{abstract}
We show that the Ricci flat Calabi's metrics on holomorphic line bundles over compact \K\--Einstein manifolds are not projectively induced. As a byproduct we solve a conjecture addressed in \cite{LSZricci} by proving that any multiple of  the Eguchi-Hanson metric on the blow-up of $\C^2$ at the origin is not projectively induced.
\end{abstract}
\section{Introduction}
It is still an open problem to classify those complex manifolds which admit a Ricci flat and projectively induced K\"ahler metric. Here a K\"ahler metric $g$ on a complex manifold $M$ (not necesseraly compact) is said to be \emph{projectively induced} if there exists a K\"ahler (isometric and holomorphic) immersion of $(M, g)$ into the complex projective space $(\mathds{C}{\rm P}^N, g_{FS})$, $N \leq +\infty$,  endowed with the Fubini--Study metric $g_{FS}$, namely the metric whose associated \K\ form is given in homogeneous coordinates  by  $\omega_{FS}=\frac{i}{2\pi}\partial\bar\partial\log (|Z_0|^2+\cdots +|Z_N|^2)$.
This problem, which represents a special case of the classical and well-studied one dealing with \K\--Einstein and projectively induced metrics, has been addressed in \cite{LSZricci} where the authors proposed the following:
\vskip 0.3cm
\noindent
Conjecture:
{\em A Ricci-flat projectively induced metric is flat.}
\vskip 0.3cm
Roughly speaking the conjecture seems plausible since  Ricci flat metrics are solutions of the Monge-Ampere equation and one should  expect that solutions to such a nonlinear PDE would be \lq\lq algebraic'' (namely projectively induced) only  in very special cases.
In \cite{LSZricci}  the conjecture is proved for radial metrics by showing that the Eguchi-Hanson metric $g_{EH}$
on the blow-up of $\C^2$ at the origin is not projectively induced (cfr. \cite[Cor. 3.3]{LSZricci}, see also \cite{MZricciflat} for not--radial cases).  It is worth pointing out that  
Ricci-flatness, even in the radial case, cannot be weakened to scalar-flatness, as shown 
by the Simanca metric (cfr. \cite[Th. 1.3]{LSZricci}, see also \cite{CAL}).

Of course a necessary condition for a \K\ metric $g$ on a complex manifold $M$
to be projectively induced is that its associated \K\ form is integral, i.e.   
$[\omega]\in H^2(M, \Z)$, where  $[\omega]$ denotes the de-Rham class of $\omega$.
Nevertheless, even if one is able to verify that  a given Ricci flat metric $g$ with associated integral \K\ form $\omega$  is not projectively induced it is not an easy task to understand if the same is true  for a (Ricci flat) metric homothetic to $g$, namely $\lambda g$.
Indeed  in the noncompact case, due for example to the fact that $\lambda \omega$ is always integral provided $M$ is contractible, the structure of the set of the positive real numbers $\lambda \in \R^+$ for which $\lambda g$ is projectively induced is in general less trivial than in the compact case (where it is always discrete). For example if  $\Omega$ is  an irreducible bounded symmetric domain endowed with its Bergman metric $g_B$ one can prove (see \cite{LZ}) that   
$(\Omega, \lambda g_B)$ admits a  \K\ immersion into $\C P^{\infty}$ if
and only if $\lambda$
 belongs to the so called  Wallach set of $\Omega$.
 Another interesting example is given by the Cigar metric $g=\frac{dz\otimes d\bar z}{1+|z|^2}$
 on $\C$ which is not projectively induced together with all its multiples $\lambda g$ (cfr. \cite{LZcigar}).
 The proof of this result is quite involved and it is based on some properties of Bell polynomials (see also Section \ref{bell} below).

The aim of this paper is to verify the validity  of the  above conjecture 
for the  Ricci flat metrics on holomorphic line bundles over compact \K\--Einstein manifolds constructed by Calabi in \cite{calabi79}. In order to state our main result (Theorem \ref{mainteor}) we briefly recall Calabi's construction.

Let $(M, g)$ be a compact  \K\--Einstein manifold  of  complex dimension $n-1$ and with associated \K\ form $\omega_g$.
Let $k_0>0$ be the Einstein constant of $g$, namely $\rho_g = k_0 \omega_g$, being $\rho_g = -\frac{i}{\pi} \partial \bar \partial \log \det(g)$ the Ricci form.
Let  $\pi: \Lambda^{n-1} M \rightarrow M$ be the canonical line bundle over $M$, with hermitian metric $h$
given, for a system of local coordinates $z = (z_1, \dots, z_{n-1})$ on an open set $U\subset M$, by:
$$h(\xi):=h(\xi dz_1 \wedge \cdots \wedge d z_{n-1}, \xi dz_1 \wedge \cdots \wedge d z_{n-1}) = \det(g)^{-1} |\xi|^2.$$
Since the metric $g$ is assumed to be \K\--Einstein one has:
 \begin{equation}\label{eqq}
\frac{i}{\pi}\partial \bar \partial \log h(\xi) = k_0 \omega_g.
 \end{equation}
 Let $u:[0, +\infty)\rightarrow\R$ be the smooth function defined by:
 \begin{equation}\label{funcal}
 u(x)=\frac{n}{k_0}\left[\left(1+cx\right)^{\frac{1}{n}}-1\right]-\sum_{j=1}^{n-1}\frac{1-\tau^j}{k_0}
 \log \left[\frac{\left(1+cx\right)^{\frac{1}{n}}-\tau^j}{1-\tau^j}\right],
 \end{equation}
 where $c>0$ is constant and   $\tau=e^{\frac {2\pi i}{n}}$.

One easily verifies that 
the function $u$  satisfies:
\begin{enumerate}
\item[$(i)$] $1 + k_0 x u'(x) > 0, \ \ u'(x) + x u''(x) > 0$,
\item[$(ii)$]
$(1 + k_0 x u'(x))^{n-1} (u'(x) + x u''(x)) = c.$
\end{enumerate}

We have the following beautiful  result due to Calabi.

\vskip 0.3cm

\noindent
{\bf Theorem C} (Calabi \cite{calabi79})
{\em Let $(M, g)$ be a compact \K--Einstein manifold of complex dimension $n-1$ and positive Einstein constant $k_0$ and $u$ the function given by \eqref{funcal}.
If  $\Phi$ is a \K\ potential for $g$, i.e. $\omega_g =\frac{i}{2\pi} \partial \bar \partial \Phi$ on  $U$, then
the function $\Psi:\pi^{-1}(U)\rightarrow \R$
defined by
\begin{equation}\label{potentialcanonical}
\Psi = \Phi \circ \pi + u\left(\det(g)^{-1} |\xi|^2\right)
\end{equation}
is a \K\  potential on $\pi^{-1}(U)$ for  a Ricci flat  and complete metric $g_C$  on the total space $\Lambda^{n-1} M$.}
\begin{remark}\label{remarmult}
Notice that if $g_C$ is the Ricci flat metric corresponding to $(M, g)$ as in the previous theorem, one easily deduces that  $\alpha g_C$, $\alpha >0$, is  the metric corresponding to 
$(M, \alpha g)$. 
\end{remark}
The previous theorem represents a special case of a more general construction due to Calabi  himself (see \cite{calabi79} for more details).
More precisely, Calabi considers a holomorphic hermitian line bundle $(L, h)\rightarrow M$    over a compact  \K-Einstein manifold $(M, g)$, with Einstein constant $k_0$ (of arbitrary sign) such  that
\begin{equation}\label{eqqq}
\Ric (h)=-\ell \omega,
\end{equation}
for some real number $\ell$,
where
 $\Ric (h)$ is the two-forms on $L$ given by:
  \begin{equation}\label{ricci}
\Ric (h)=-\frac{i}{\pi}
\partial\bar\partial\log h(\sigma(x), \sigma (x)),
\end{equation} 
for a trivializing holomorphic section
$\sigma$ of $L$.

 Then, Calabi shows that there exists a \K-Einstein metric $\tilde g$ with Einstein constant $k_0-
 \ell$ on  an open subset $V$ of $L$, Moreover, when $\tilde g$ is Ricci flat, i.e. $\ell=k_0$, then $\tilde g$ is defined on the whole $L$, i.e.  $V=L$.
Notice also that   $\tilde g$ is constructed in such a way that the natural  inclusion  $M\hookrightarrow L$ is  a \K\ immersion, 
i.e. $\tilde g _{|M}=g$. 

We can now state our main result.
 \begin{theorem}\label{mainteor}
Let $(M, g)$ be a compact K\"ahler--Einstein manifold with Einstein constant $k_0$. 
The Calabi's metric $\tilde g$ on $L$ cannot be simultaneously  Ricci flat and projectively induced. \end{theorem}

When the manifold $(M, g)$  is assumed to be   homogeneous, i.e. a flag manifold,  Theorem \ref{mainteor} gives  a large family of 
Ricci flat metrics  which are not projectively induced (the reader is referred to \cite{LMZ} for an explicit description of the \K\ metrics and their potentials on flag manifolds). 
In particular,  when   $(M, g)=(\C P^1, g_{FS})$ 
then the Calabi's Ricci flat metric  $g_C$ on $O(-2)=\Lambda ^1\C P^1$ is  the celebrated  Eguchi-Hanson metric $g_{EH}$ on the blow-up of $\C^2$ at the origin (see \cite{calabi79}, \cite{LSZricci} and references therein).
 Then, by   Remark \ref{remarmult} one obtains the following corollary of Theorem \ref{mainteor}, which thereby solves a question raised in \cite{LSZricci}.

\begin{corollary} 
The metric  $mg_{EH}$ is not projectively induced for any positive integer $m$.
\end{corollary}

The proof of  Theorem \ref{mainteor} is given in Section \ref{proofmainteor}.
Roughly speaking we first show that if  the metric $\tilde g$ on $L$ 
is projectively induced then  $\tilde g=g_C$ (and $L=\Lambda^{n-1}M$).
Then we use  the structure of  Calabi's diastasis function $D_{g_C}$ for the metric $g_C$
and its link with the diastasis of the metric $g$ (to whom  Section \ref{criterion} is dedicated)
to show that $g_C$ is not projectively induced. Finally, in Section \ref{numerical} we prove an inequality on Bell polynomials (Theorem \ref{teornum}) based on the fact that the metric $mg_C$ is not projectively induced.  We were not able  to  find a direct proof 
of Theorem \ref{teornum}.  Nevertheless, we believe that the interplay between Bell polynomials and projectively induced metrics should deserve further investigation.

\section{Calabi's diastasis function for the metric $g_C$}\label{criterion}

Let $(X, G)$ be an $n$-dimensional real analytic K\"ahler manifold and denote by $\Omega$ the K\"ahler form associated to $G$. Set local coordinates $z=(z_1,\dots, z_n)$ on a coordinate chart $U\subset X$ and denote by $\varphi\!:U\rightarrow \R$ a K\"ahler potential for $G$ on $U$, i.e. $\Omega|_U=\frac i2\partial \bar\partial \varphi$. 
Calabi's  diastasis function  \cite{calabi} for $G$ on $U$ is given by:
\begin{equation}
D_G(z,z')=\tilde\varphi(z,\bar z)+\tilde\varphi(z',\bar z')-\tilde\varphi(z,\bar z')-\tilde\varphi(z',\bar z),\nonumber
\end{equation}
where $\tilde\varphi$ is the analytic continuation of $\varphi$ on a neighborhood of the diagonal of $U\times U$. It is easy to see that the diastasis is symmetric in $z$, $z'$ and that, once fixed one of its two entries, it is a K\"ahler potential for $G$. In particular, we will denote by $D_G(z):=D_G(0,z)$ the diastasis centered at the origin of the coordinate system. 
Among the other K\"ahler potentials, the diastasis function is characterized by the fact that in every coordinate system $(z)$ centered in $p$, the $\infty\times\infty$ matrix of coefficients $(a_{jk})$ in its power expansion in terms of $z$ and $\bar z$ around the origin:
\begin{equation}\label{powexdiastc}
D_G(z)=\sum_{j,k=0}^{\infty}a_{jk}z^{m_j}\bar z^{m_k},
\end{equation}
satisfies $a_{j0}=a_{0j}=0$ for every nonnegative integer $j$,
i.e it does not contain either holomorphic or antiholomorphic terms.
 In the multi-index notation
the $m_j$'s are $n$-tuples of integers arranged in lexicographic order.  

In order to prove Theorem \ref{mainteor} we recall Calabi's  criterion 
for a \K\ metric to admit a local  {\em K\"ahler} 
immersion into a finite or infinite dimensional complex projective space $(\mathds{C}{\rm P}^N, g_{FS})$, $N\leq\infty$, through the diastasis function (the reader is referred  to \cite{diastbook} for a more detailed and updated exposition of the subject). 
 
 \begin{lemma}(Calabi's criterion)\label{criterium}
Let $(X, G)$ be a real analytic \K\ manifold and let $D_G$ be its diastasis function around the origin. Then $(X, G)$ admits a local K\"ahler immersion into $(\C P^N, g_{FS})$ if and only if the $\infty\times\infty$ matrix of coefficients $(b_{jk})$ in the power expansion with respect to $z$ and $\bar z$:
\begin{equation}\label{powexdiastcp}
e^{D_G(z)}-1=\sum_{j,k=0}^{\infty}b_{jk}z^{m_j}\bar z^{m_k},
\end{equation}
is positive semidefinite.
\end{lemma}

In the following lemma we describe Calabi's diastasis function for the metric $g_C$
on $\Lambda^{n-1}M$ given in Theorem C.
In order to do that 
we need to introduce the concept of
Bochner's coordinates (cfr. \cite{boc}, \cite{calabi}).
Given a real analytic
\K\ metric $G$ on $X$
and a point $p\in X$,
one can always find local 
(complex) coordinates
in a neighborhood of 
$p$
such that
$$D_G(z)=|z|^2+
\sum _{|j|, |k|\geq 2}
a_{jk}z^j\bar z^k.$$
These coordinates,
uniquely defined up
to a unitary transformation,
are called 
{\em the Bochner's coordinates}
with respect to the point $p$.

\begin{lemma}\label{senzanome}
Let $(M, g)$ be a compact  \K-Einstein manifold with positive Einsten constant $k_0$
and $D=D_g:U\rightarrow \R$ be its diastasis function centered at the origin.
Then the diastasis function for the metric   $g_C$ in Bochner's coordinates reads as:
\begin{equation}\label{potentialcanonicaldiast}
 D_{g_C}=D+ u(e^{\frac{k_0}{2} D} |\xi|^2).
\end{equation}
Moreover, the matrix of coefficients $(b_{il})$ given on \eqref{powexdiastcp} associated to $D_{g_C}$ is a block matrix where each block $(b_{jk}^r)$, $r=0,1,2,\dots$, is given by 
\begin{equation}\label{eqfond}
(b_{jk}^r)=h_r(u)(c^r_{jk}),
\end{equation} 
where $(c^r_{jk})$ is the matrix of coefficients \eqref{powexdiastcp} associated to $(r\frac{k_0}{2}+1)D$ and $h_r(u)$ is a constant that depends on the derivatives of $u(x)$ evaluated at $x=0$ up to the $r$-th order.
In particular:
\begin{equation}\label{h1}
h_1(u)=u'(0);
\end{equation}
\begin{equation}\label{h2}
h_2(u)= \frac12 \left(u'(0)^2 + u''(0)\right);
\end{equation}
\begin{equation}\label{h3}
h_3(u)=\frac1{6}\left(u'(0)^3 +3u''(0)u'(0)+u'''(0)\right);
\end{equation}
\begin{equation}\label{h4}
h_4(u)=\frac1{24} \left(u'(0)^4 +6u''(0)u'(0)^2+4u'''(0)u'(0)+3u''(0)^2+u^{(iv)}(0)\right).
\end{equation}

\end{lemma}
\begin{proof}
 Let  $z = (z_1, \dots, z_{n-1})$ be Bochner coordinates on a open set of $U\subset M$. 
The condition $\rho_g = k_0\omega_g$ reads:
$$
-i \partial \bar \partial \log \det \left( \frac{\partial^2 D}{\partial z_i \partial \bar z_j} \right) = \frac{k_0}{2}i \partial \bar \partial D,
$$
that is:
$$
\log \det \left( \frac{\partial^2 D}{\partial z_i \partial \bar z_j} \right) = -\frac{k_0}{2} D + f+ \bar f,
$$
where $f$ is a holomorphic function on $U$.
By using the condition that $z$ are Bochner coordinates and $D$ is the diastasis, it is not hard to see (cfr. also \cite{ALZampa}) that $f= 0$, and hence we can write:
$$
\det(g) = \det \left( \frac{\partial^2 D}{\partial z_i \partial \bar z_j} \right) = e^{-\frac{k_0}{2} D}.
$$
By \eqref{potentialcanonical} it follows that \eqref{potentialcanonicaldiast} is a \K\ potential for $g_C$ on $\pi^{-1}(U)$. Moreover, since the expansion  \eqref{powexdiastc}
of  $D$ does not contain holomorphic or antiholomorphic terms one easily  sees that  the same is true for the expansion of 
 $D_{g_C}$ with respect to the coordinates  $z_1, \dots, z_{n-1}, \xi$.
 
 In order to prove the second assertion observe first that 
 since $D_{g_C}$ depends on $\xi$, $\bar \xi$ only through $|\xi|^2$, all the derivatives taken a different number of times with respect to $\xi$ than $\bar \xi$, vanish once evaluated at $\xi=0$. Thus, the nonzero entries of $(b_{il})$ are of the form: 
$$
\frac{1}{m_j!m_k!}\frac{\partial^{|m_j|+|m_k|}}{\partial z^{m_j}\partial \bar z^{m_k}}\left[\frac{1}{r!^2}\frac{\partial^{2r}}{\partial \xi^r\partial \bar \xi^r}\left(e^{D+ u(e^{\frac{k_0}{2} D} |\xi|^2)} - 1\right)|_{\xi=0}\right]_{z=\bar z=0},
$$ 
and, for each $r=0,1,2,\dots$, we have a block submatrix $(b^r_{jk})$ of $(b_{il})$ defined by:
\begin{equation*}
\begin{split}
b^r_{jk}:=&\frac{1}{m_j!m_k!}\frac{\partial^{|m_j|+|m_k|}}{\partial z^{m_j}\partial \bar z^{m_k}}\left[\frac{1}{r!^2}\frac{\partial^{2r}}{\partial \xi^r\partial \bar \xi^r}\left(e^{D+ u(e^{\frac{k_0}{2} D} |\xi|^2)} - 1\right)|_{\xi=0}\right]_{z=\bar z=0}\\
=&h_r(u)\frac{1}{m_j!m_k!}\frac{\partial^{|m_j|+|m_k|}}{\partial z^{m_j}\partial \bar z^{m_k}}e^{(r\frac{k_0}{2}+1)D}|_{z=\bar z=0}\\
=&h_r(u)c^r_{jk},
\end{split}
\end{equation*}
where:
$$
c^r_{jk}:=\frac{1}{m_j!m_k!}\frac{\partial^{|m_j|+|m_k|}}{\partial z^{m_j}\partial \bar z^{m_k}}e^{(r\frac{k_0}{2}+1)D}|_{z=\bar z=0},
$$
and $h_r(u)$ is a constant that depends on the derivatives up to  $r$-th order on $u(x)$.
Thus \eqref{eqfond} is proved.
Finally, \eqref{h1}, \eqref{h2}, \eqref{h3} and \eqref{h4} follows by  straightforward computations.
 \end{proof}

\section{Proof of Theorem \ref{mainteor}}\label{proofmainteor}
Let us begin with the following lemma.
\begin{lemma}\label{neccond}
Let $(M,g)$ be a compact K\"ahler--Einstein manifold with Einstein constant $k_0$. If the  Ricci flat Calabi's metric
$g_C$ on $\Lambda^{n-1}M$  is projectively induced then the following conditions hold true.
\begin{itemize}
\item [(a)]
 $\frac{k_0}{2}$ is a positive integer;
 \item [(b)]
 $(r\frac{k_0}{2}+1)g$ is projectively induced, for any $r=0, 1\dots$;
 \item [(c)]
 each $h_r(u)$ given by \eqref{eqfond} is non negative, for any $r=0, 1\dots$. 
 \end{itemize}
 \end{lemma}
\begin{proof}
The assumption that $g_C$ is projectively induced and Calabi's criterion (Lemma \ref{criterium}) imply that the block submatrices $(b_{jk}^r)=h_r(u)(c^r_{jk})$ given by \eqref{eqfond} are positive semidefinite, for any $r=0, 1\dots$. Since  $g_C{_{{|M}}}=g$ it follows that also $g$ is projectively induced, and so  (b) is valid   for $r=0$ . When $r=1$ (cfr. \eqref{h1}), $h_1(u)=u'(0)e^{u(0)}=c>0$ (where $c$ is the constant appearing in $(ii)$ before Theorem C) which proves that $(c^1_{jk})$ is positive semidefinite.  Again by Calabi's criterion 
one deduces that  $(\frac{k_0}{2}+1)g$ is projectively induced. 
If $\omega$ is the \K\ form associated to $g$ the fact that both  $g$ and
$(\frac{k_0}{2}+1)g$ are  projectively induced imply that  $\omega$ and $(\frac{k_0}{2}+1)\omega$ are  integral forms, forcing $\frac{k_0}{2}$ to be a positive integer, i.e. (a).
Hence $(r\frac{k_0}{2}+1)$ is a positive integer and (b) follows by composing the \K\ immersion inducing $g$ with a suitable normalization of the Veronese embbedding (cfr. \cite[Theorem 13]{calabi}). Finally, by combining (b), \eqref{eqfond} and Calabi's criterion we get (c).
\end{proof}

\begin{proof}[Proof of Theorem \ref{mainteor}]
We start noticing that if  the metric $\tilde g$ on $L$ is Ricci flat and projectively induced
then $\tilde g=g_C$ and $L=\Lambda^{n-1}M$.
Indeed, since $\tilde g _{|M}=g$  one has  that also  $g$ is  projectively induced and, by a result  of Hulin \cite{hu} $(M, g)$ has positive Einstein constant $k_0$.
Therefore, $M$ is simply-connected,  and the Ricci flatness of $\tilde g$, i.e $\ell=k_0$,
combined with \eqref{eqqq}  yields
$$c_1(L)=[-\frac{i}{2\pi}\partial\bar\partial\log h]=-\frac{k_0}{2}[\omega]=-[\frac{\rho}{2}]=c_1(\Lambda^{n-1}M),$$
 and so $L$ is holomorphically equivalent to the canonical bundle
$\Lambda^{n-1} M$.

Thus, assume by contradiction that the metric $g_C$ is projectively induced.
We will show that the sign of some $h_r(u)$ is negative and by (c) of  Lemma \ref{neccond} this gives the desired contradiction.
Setting $x=0$ in  equation $(ii)$ before Theorem C, one has $u'(0) = c$, while differentiating the same equation with respect to $x$ one gets:
\begin{equation}\label{firstderiv}
(n-1) k_0 (1 + k_0 x u')^{n-2} (u' + x u'')^2 + (1 + k_0 x u')^{n-1}(2 u'' + x u''') = 0
\end{equation}
which, evaluated in $x=0$, gives:
$$
u''(0) = - \frac{(n-1)k_0}{2}u'(0)^2.
$$
Combining this with $u'(0)=c$, one deduces:
\begin{equation}\label{condizNE}
u''(0) + u'^2(0) = \left( 1 - \frac{(n-1)k_0}{2} \right) c^2.
\end{equation}

 Recalling that $c>0$ and that, by (a) of Lemma \ref{neccond}, $\frac{k_0}{2}$ is a positive integer, we deduce  by \eqref{h2} that $h_2(u)$  is negative except for $k_0=2$ and $n=2$ (where it vanishes).
 For this values  we will show that  $h_4(u)<0$.
In order to do that, we differentiate \eqref{firstderiv} twice and  evaluating at $u'(0)=c$ and
$u''(0)=-c^2$ we get
$u'''(0)=4c^3,$
and 
$u^{(iv)}(0)=-30c^4.$
Plugging these into \eqref{h4} gives that $h_4(u)=-\frac{2}{3}c^4 <0$,
and we are done\footnote{The reason why we look at the coefficient $h_4(u)$
is because by \eqref{h3} one deduces $h_3(u)=\frac{c^3}{3}>0$.}.
\end{proof}

\begin{remark}\label{flagex}
It is an open problem to classify the compact  \K --Einstein manifolds admitting  \K\ immersions into complex projective spaces. 
The only known examples are indeed flag manifolds. The reader is referred to \cite[Ch.3]{diastbook} for other properties of projectively induced homogeneous metrics and, in particular,  for the 
proof that any integral \K\ form on a compact flag manifold is projectively induced.
Combining this  result with the fact that on any flag manifold one can find a \K--Einstein integral form with Einstein constant $k_0=2$ (see \cite{LMZ17}) we deduce the existence of a projectively induced \K--Einstein metric satisfying (a) in Lemma \ref{neccond}. Actually by taking the \K\ product of such manifolds
one can construct \K--Einstein metrics with arbitrary large even value of their Einstein constant.
This means that also if one restricts to the case of flag manifolds the proof of Theorem
\ref{mainteor} 
cannot be deduced by simply showing that condition (a) is not satisfied.
\end{remark}

\section{An application to a numerical problem}\label{numerical}
Recall that the partial (exponential) Bell polynomials $B_{r,j}(x):=B_{r,j}(x_1,\dots, x_{r-j+1})$ of degree $r$ and weight $j$ are defined by (see e.g. \cite[p. 133]{comtet}):
\begin{equation}\label{bjkdef}
B_{r,j}(x_1,\dots, x_{r-j+1})=\sum\frac{r!}{s_1!\cdots s_{r-j+1}!}\left(\frac{x_1}{1!}\right)^{s_1}\left(\frac{x_2}{2!}\right)^{s_2}\cdots \left(\frac{x_{r-j+1}}{(r-j+1)!}\right)^{s_{r-j+1}},
\end{equation}
where the sum is taken over the integers solutions of:
$$
\begin{cases}s_1+2s_2+\dots+js_{r-j+1}=r\\ s_1+\dots+s_{r-j+1}=j.\end{cases}
$$

The complete Bell polynomials are given by:
$$
Y_r(x_1,\dots, x_r)=\sum_{j=1}^rB_{r,j}(x),\quad Y_0:=0,
$$
and the role they play in our context is given by the following formula \cite[Eq. 3b, p.134]{comtet}:
\begin{equation}\label{exp}
\frac{d^r}{dx^r}\left(\exp\left(\sum_{k=1}^\infty a_k\frac{x^k}{k!}\right)\right)|_0=Y_r(a_1,\dots, a_r).
\end{equation}

We can prove the following inequality. 

\begin{theorem}\label{teornum}
Let $n\in \mathds Z^+$, $n\geq 2$. For any $q\in \mathds Q^+$, there exists $r$ sufficiently large such that the following inequality holds true:
\begin{equation}\label{risultato}
(-1)^r \sum_{j=1}^{r}\left(-q\right)^j B_{r,j}\!\!\left(1,\frac{n-1}2,\frac{(n-1)(2n-1)}{3},\dots, \frac{ 1}l\prod_{s=1}^{l-1}(ns-1),\dots\right)<0.
\end{equation}
\end{theorem}
\begin{proof}
Let  $u(x)$ be the function given by \eqref{funcal}.
Then a straightforward computation by differentiating
$(ii)$ before Theorem C, shows that its series expansion is given by:
$$
u(x)=\sum_{k=1}^\infty \frac{a_k}{k!}x^k,\quad a_1=c,
$$
where:
\begin{equation}\label{ajexpl}
a_j=\frac{(-1)^{j+1}}jc^jk_0^{j-1}\prod_{s=1}^{j-1}(ns-1), \quad j=2,3,4,\dots.
\end{equation}
Let $m>0$. By \eqref{exp} it is easy to see that:
\begin{equation}\label{rder}
\begin{split}
\frac{\partial^{2r}}{\partial \xi^r \partial \bar \xi^r} [e^{mD+ mu(e^{\frac{k_0}{2}D} |\xi|^2)} - 1]|_{\xi = 0}=&r!e^{(r\frac{k_0}{2}+m) D}\sum_{j=1}^rm^jB_{r,j}(a_1,a_2,\dots),
\end{split}
\end{equation}
and in particular, the block matrices described in \eqref{eqfond} are now given by:
$$
(b_{jk}^{r,m})=h_r(u,m)(c^{r,m}_{jk}),
$$
where:
$$
h_r(u,m):=r!\sum_{j=1}^rm^jB_{r,j}(a_1,a_2,\dots),\quad c^{r,m}_{jk}:=\frac{1}{m_j!m_k!}\frac{\partial^{|m_j|+|m_k|}}{\partial z^{m_j}\partial \bar z^{m_k}}e^{(r\frac{k_0}{2}+m)D}|_{z=\bar z=0},
$$
(in this notations, the constants $h_r(u)$ appearing in \eqref{eqfond} are given by $h_r(u)=h_r(u,1)=r!Y_r(a_1,\dots,a_r)$). By Theorem \ref{mainteor} and Remark \ref{remarmult} the metric $mg_C$ on 
$\Lambda^{n-1}M$ is not projectively induced for any $m>0$. By Calabi's criterion this implies that each block $(b_{jk}^{r,m})$ is not semipositive definite. If we assume $(M,g)$ to be projectively induced and $r\frac{k_0}{2}+m$ to be a positive integer for any positive integer $r$, i.e. $\frac{k_0}{2}$, $m\in \mathds Z^+$, then Calabi's criterion applied to $(M,(r\frac{k_0}{2}+m)g)$ implies that each block $(c^{r,m}_{jk})$ is semipositive definite. Thus, at least one of the constants $h_r(u,m)$ is forced to be negative, i.e. for any $m\in \mathds Z^+$ there exists a sufficiently large $r$ such that: 
\begin{equation*}\label{Brj}
\sum_{j=1}^rm^jB_{r,j}(a_1,a_2,\dots)<0.
\end{equation*}
Since:
$$
B_{r,j}(tpx_1,tp^2x_2,\dots,tp^{r-j+1}x_{r-j+1})=t^jp^rB_{r,j}(x_1,\dots, x_{r-j+1}),
$$
by \eqref{ajexpl} we can write:
\begin{equation*}\label{bell}
\begin{split}
\sum_{j=1}^rm^j&B_{r,j}(a_1,a_2,\dots)=(-ck_0)^r\sum_{j=1}^r\left(-\frac{m}{k_0}\right)^j\cdot\\
&\cdot B_{r,j}\!\!\left(1,\frac{n-1}2,\frac{(n-1)(2n-1)}{3},\dots, \frac{ 1}l\prod_{s=1}^{l-1}(ns-1),\dots\right).
\end{split}
\end{equation*}
At this point, \eqref{risultato} follows by observing that one can construct examples of projectively induced K\"ahler--Einstein manifolds with any positive integer value of $k_0/2$ (see Remark \ref{flagex}).
\end{proof}

\end{document}